\documentclass{amsart}
\usepackage{graphicx}
\usepackage{amscd}
\usepackage{amsmath}
\usepackage{amsfonts}
\usepackage{amssymb}
\theoremstyle{plain}
\newtheorem{theorem}{Theorem}
\newtheorem{corollary}[theorem]{Corollary}
\newtheorem{lemma}[theorem]{Lemma}
\newtheorem{proposition}[theorem]{Proposition}
\theoremstyle{definition}

\newtheorem{definition}[theorem]{Definition}

\newtheorem{question}[theorem]{Question}

\theoremstyle{remark}

\begin{document}
\title{Domains with invertible-radical factorization}

\author{Malik Tusif Ahmed and Tiberiu Dumitrescu}
\address{Abdus Salam School of Mathematical Sciences GCU Lahore, Pakistan}
\email{tusif.ahmed@sms.edu.pk, tusif.ahmad92@gmail.com (Ahmed)}
\address{Facultatea de Matematica si Informatica,University of Bucharest,14 A\-ca\-de\-mi\-ei Str., Bucharest, RO 010014,Romania}
\email{tiberiu@fmi.unibuc.ro, tiberiu\_dumitrescu2003@yahoo.com (Dumitrescu)}

\begin{abstract}
We study those integral domains in which every   proper ideal  can be written as an invertible ideal multiplied by a nonempty  product of proper radical ideals.
\end{abstract}


\maketitle

In \cite{VY} Vaughan and Yeagy introduced and studied the notion of {\em SP-domain}, i.e. an integral domain whose ideals are products of radical (also called semiprime) ideals.
They proved that an SP-domain is always almost Dedekind  (i.e. every localization at a maximal ideal is a rank one discrete valuation domain (DVR)). They also gave an example of an SP-domain which is not Dedekind. For examples of almost Dedekind domains which are not SP, see \cite{Y} and \cite[Example 3.4.1]{FHL}.
The study  of SP-domains was continued by Olberding  (in \cite{O}) who gave several characterizations for SP-domains inside the class of almost Dedekind domains and also gave a method to construct SP-domains starting from Boolean topological spaces.

 In a sequence of papers (\cite{Og}, \cite{OI}, \cite{OII}) Olberding introduced and studied the concept of {\em ZPUI (Zerlegung Prim und Umkehrbaridealen) domain}, i.e. a domain for which every proper nonzero ideal can be factored  as a product of an invertible  ideal times a nonempty  product of pairwise comaximal prime ideals (Olberding did his study for commutative rings, but we are interested here only in  domain case). 
 He showed that a domain $A$ is  ZPUI if and only if every proper nonzero ideal can be factored  as a product of a finitely generated ideal times a nonempty finite product of  prime ideals if and only if $A$ is a strongly discrete h-local Pr\"ufer domain \cite[Theorem 1.1]{OII}. Let $A$ be a domain. We recall that $A$ is {\em h-local} if the factor ring $A/I$ is local (resp. semilocal) for each nonzero prime ideal (resp.  nonzero ideal) $I$ of $A$.   Also $A$ is a {\em Pr\"ufer domain} if its nonzero finitely generated ideals are  invertible. A  Pr\"ufer domain is {\em strongly discrete} if it has no idempotent prime ideal except zero.

In this paper we study a new class of domains. Call a  domain $A$ an {\em ISP-domain (invertible semiprime domain)} if each proper ideal of $A$ is can be written as an invertible ideal multiplied by a nonempty  product of proper radical ideals. So any SP-domain (resp. ZPUI-domain) is an ISP-domain.

 In Section 1 we prove the following results.
If $A$ is an ISP-domain, then any factor domain of $A$ and any (flat) overring of $A$ are also ISP-domains (Propositions \ref{1} and \ref{2}, see also Proposition \ref{11}).  Any one-dimensional ISP-domain is almost Dedekind and, consequently, any Noetherian ISP-domain is a Dedekind domain (Corollary \ref{12}).
In Section 2 we prove that if $A$ is an ISP-domain, then $A$ is a strongly discrete Pr\"ufer domain and every nonzero  prime ideal of $A$ is contained in a unique maximal ideal (Theorem \ref{7}). Consequently, an ISP-domain such that every ideal has finitely many minimal prime ideals is a ZPUI-domain (Corollary \ref{14}).
In Section 3 we consider  the question whether every one-dimensional ISP-domain is an SP-domain. We provide a positive answer for domains in which
every nonzero element  is contained in at most finitely many noninvertible maximal ideals (Theorem \ref{3}). In particular, a one-dimensional ISP-domain having only finitely many noninvertible maximal ideals  is an  SP-domain (Corollary \ref{21}).
In Section 4 we give an example of a two-dimensional ISP-domain $A$ which is not h-local. Hence $A$ is neither an SP-domain nor a ZPUI-domain.

Throughout this paper, our rings are commutative and unitary. For any undefined terminology, we refer the reader to \cite{G} or \cite{K}.

\section{Basic results}

We recall the key definition of our paper.

\begin{definition}\label{13}
We say that a domain $A$ is an {\em ISP-domain (invertible semiprime domain)} if every proper nonzero ideal  $I$ of $A$ can be written as $JQ_1\cdots Q_n$ where $n\geq 1$, $J$ is an invertible ideal and each $Q_i$ is a  proper  radical ideal.
\end{definition}

Clearly a ZPUI-domain or an SP-domain is an ISP-domain.  
%
 The well-known Bezout domain $A=\mathbb{Z}+X\mathbb{Q}[X]$ (see  \cite{CMZ} for its basic properties) is not an ISP-domain. Indeed, consider the ideal $I=X\mathbb{Z}[1/2]+X^2\mathbb{Q}[X]$. The radical ideals containing $I$ are $X\mathbb{Q}[X]$ and $nA=n\mathbb{Z}+X\mathbb{Q}[X]$ with $n$ a positive square-free integer. So there is no element $f\in A$ such that $I\subseteq fA$ and $If^{-1}$ is a product of radical ideals. Note that every proper nonzero principal ideal $gA$ can be written in the form required by Definition \ref{13}. Indeed, if $g\not\in X\mathbb{Q}[X]$, then $g$ is a product of principal primes and if  $g\in X\mathbb{Q}[X]$, then $g=2(g/2)A$. Note also that $A$ is strongly discrete. 
\\[1mm]
In this section we prove a few basic properties of ISP-domains.

\begin{proposition}\label{1}
 If $A$ is an  ISP-domain and $P$ a  prime ideal of $A$, then $A/P$ is an ISP-domain.
 \end{proposition}
\begin{proof}
Let  $I\supset P$ be a proper ideal of $A$. As $A$ is an ISP-domain, we can write     $I=JH_1\cdots H_n$ with $J$ an invertible ideal, $n\geq 1$ and each  $H_i$ a  proper  radical ideal. Since all ideals $I,H_1,...,H_n$ contain $P$, we get $I/P=(J/P)(H_1/P)\cdots (H_n/P)$ with $J/P$ invertible and each  $H_i/P$ a proper radical ideal.
\end{proof}

\begin{proposition}\label{2} 
Let $A$ be an ISP-domain and $B$ a flat overring of $A$. Then $B$ is an ISP-domain.
\end{proposition}
\begin{proof}
Let  $H$ be  a proper nonzero ideal of $B$ and $I=H\cap A$. 
By \cite[Theorem 2]{Ak}, $IB=H$.  As $A$ is an ISP-domain, we can write   $I=JQ_1\cdots Q_n$ with $J$ an invertible ideal, $n\geq 1$ and all  $Q_i$'s   proper  radical ideals.
Then $H=IB= (JB)(Q_1B)\cdots (Q_nB)$, where $JB$ is invertible  and each $Q_iB$ is  a radical ideal. 
Indeed, since $A_{M\cap A}=B_M$ for every $M\in Max(B)$ (cf. \cite[Theorem 2]{Ak}), 
it is easy to check locally that a radical ideal of $A$ extends to a radical ideal of $B$. If every $Q_iB$  is equal to $B$, then  $H=JB$ and $WB=B$ where $W=Q_1\cdots Q_n$. Hence $J\subseteq JB\cap A=H\cap A=I=JW\subseteq J$, so $J=JW$, thus  $W=A$ (because $J$ is invertible), a contradiction. 
\end{proof}

We give a simple application of Proposition \ref{2}.

\begin{corollary} \label{12}
Any one-dimensional ISP-domain is almost Dedekind. Consequently, a Noetherian ISP-domain is a Dedekind domain.
\end{corollary}
\begin{proof}
Let $A$ be a one-dimensional ISP-domain.
By Proposition \ref{2}, we may assume that $A$ is local with maximal ideal $M$.   Let $x\in M-\{0\}$. Since the radical ideals of $A$ are $0$ and $M$, we get $xA=yM^k$ for some $y\in A$ and  $k\geq 1$, so $M$ is invertible, hence  $A$ is a DVR.
For the ``Consequently'' part, assume, by the contrary, that $A$ is a Noetherian ISP-domain which is not Dedekind. By the first part, $dim(A)\geq 2$, so, using Proposition \ref{2}, we may assume that $A$ is a two-dimensional local domain (with maximal ideal $M$). Let $x\in M-M^2$, $P$ a height one prime ideal containing $x$ and let $y\in M-P$. Since $P\not\subseteq M^2$, $M$ is minimal over $(P,y^2)$ and $A$ is an ISP-domain, we get  $(P,y^2)=M$. Modding out by $P$, we get a contradiction.
\end{proof}




\section{ISP domains are Pr\"ufer strongly discrete}

The following theorem is the main result of this paper.

\begin{theorem}\label{7} 
If $A$ is an ISP-domain, then 

$(a)$ $A$ is a strongly discrete Pr\"ufer domain, and 

$(b)$ every nonzero  prime ideal of $A$ is contained in a unique maximal ideal.
\\
In particular, a local domain is an ISP-domain if and only if it is a strongly discrete valuation domain.
\end{theorem}

We need a string of three lemmas.




\begin{lemma}\label{4}
If $A$ is an ISP-domain and $P\subset M$ are nonzero prime ideals of $A$, then $P\subseteq M^2A_M$.
\end{lemma}
\begin{proof}
By Proposition  \ref{2}, we may assume that $A$ is local with maximal ideal $M$. Assume that $P\not\subseteq M^2$ and take $x\in M-P$. Since $A$ is an ISP-domain and $P\not\subseteq M^2$, we get that $(P,x^2)$ is a radical ideal, so $(P,x^2)=(P,x)$ which gives a contradiction after modding out by $P$.
\end{proof}

\begin{lemma}\label{5}
Let $A$ be an ISP-domain, $P \subset M$ prime ideals and $x\in M-P$ such that $M$ is minimal over $(P,x)$. Then $MA_M$ is a principal ideal.
\end{lemma}
\begin{proof}
By Proposition \ref{2}, we may assume that $A$ is local with maximal ideal $M$. We show first that $M$ is not idempotent. On contrary assume that $M^2=M$. Note that $\sqrt{(P,x)}=M$ is the only radical ideal containing $(P,x)$. As $A$ is an ISP-domain and $M=M^2$, we get $(P,x)=yM$ for some $y \in A$. As $P\subseteq yM$, we get $y  \notin P$ (otherwise $P=yA\subseteq yM$), hence $P=Py$.
From $x \in yM$, we get  $x=yz$ for some $z \in M$. Now from $(Py,yz)=yM$, we get $(P,z)=M$, so $M/P$ is a principal idempotent nonzero maximal ideal of $A/P$, a contradiction.
Thus $M$ is not idempotent and let us pick $w \in M-M^2$. By Lemma \ref{4}, $M$ is the  only prime ideal containing $w$, so $wA=M$ because $A$ is an ISP-domain.
\end{proof}

\begin{lemma}\label{6}
If $A$ is an ISP-domain and $I$ an invertible radical proper ideal of $A$, then $A/I$ is zero-dimensional. 
\end{lemma}
\begin{proof} 
On contrary assume that \em{dim}$(A/I)\geq 1$. Then there exist two prime ideals $P\subset M$ and  $x \in M-P$ such that 
$I \subseteq P$   and $M$ is minimal over $(P,x)$. By Lemma \ref{5}, $MA_M$ is principal. Localizing at $M$, we may assume that $A$ is local with maximal ideal $M$. Then $I=yA$ and $M=zA$ for some $y,z \in A$. As $I\subset M$,
we get $y=az^2$ for some $a\in A$, so  $az \in \sqrt{yA}=yA$, hence $y=az^2\in yzA$, thus $1\in zA=M$, a contradiction.
\end{proof}


\vspace{2mm}
{\em Proof of Theorem \ref{7}}.
$(a)$ By \cite[Lemma 3.2]{OII},  it suffices 
to show that $PA_P$ is a principal ideal for every nonzero prime  ideal $P$ of $A$. Set $B=A_P$ and $M=PA_P$. 
By   Proposition  \ref{2}, $B$ is an ISP-domain.
Given  $x \in M-\{0\}$, we write $xB=yH_1\cdots H_n$ with $y\in B$,  $n\geq 1$  and $H_i$ a proper  radical ideal for $i=1$ to $n$. Then each $H_i$ is invertible hence principal, because $B$ is local.
By  Lemma \ref{6}, we have $Spec(B/H_1)=\{M/H_1\}$, hence $H_1=\sqrt{H_1}=M$.

$(b)$ By Proposition \ref{2},  we may assume that $A$ is semilocal. Indeed, if 
$M_1$ and $M_2$ are two distinct maximal ideals containing a nonzero prime ideal, then $(b)$ fails for $A_S$, where $S=A-(M_1\cup M_2)$.
Now let $I$ be a nonzero radical ideal. Since 
$A$ is a semilocal Pr\"ufer domain, it follows that $I$ has finitely many minimal primes, say $P_1$,...,$P_n$. Then $I=P_1\cap\cdots \cap P_n=P_1\cdots P_n$ because $P_1$,..,$P_n$ are incomparable prime ideals in a   Pr\"ufer domain, hence pairwise comaximal. Since $A$ is an ISP-domain and every nonzero radical ideal  is a product of primes, $A$ is a ZPUI-domain. By \cite[Theorem 1.1]{OII}, $A$ is h-local, so $(b)$ holds. 
The ``in particular'' assertion follows from \cite[Theorem 1.1]{OII}.
\hspace{5mm} $\square$\\[1mm]

We give two corollaries of Theorem \ref{7}.


\begin{corollary}\label{11}
 Any overring of an ISP-domain is also an ISP-domain.
\end{corollary}
\begin{proof}
Let $A$ be an ISP-domain and $B$ an overring of $A$.  By Theorem \ref{7}, $A$ is a   Pr\"ufer domain, so $B$ is $A$-flat, cf. \cite[page 798]{R}. Apply Proposition \ref{2}.
\end{proof}

\begin{corollary}\label{14}
 For a domain $A$, the following  are equivalent.
 
 $(a)$ $A$ is a ZPUI-domain.
 
 $(b)$ $A$ is an h-local strongly discrete Pr\"ufer domain.
 
  $(c)$ $A$ is an h-local ISP-domain.
 
 $(d)$ $A$ is a generalized Dedekind ISP-domain.
 
 $(e)$ $A$ is an ISP-domain such that $Min(I)$ is finite for each  ideal $I$.
\end{corollary}
\begin{proof} 
$(a)$ $\Leftrightarrow$ $(b)$ is a part of \cite[Theorem 1.1]{OII}. 
Implications $[(a)$ and $(b)]$ $\Rightarrow$ $(c)$ $\Rightarrow$ $(d)$ $\Rightarrow$ $(e)$ are well-known.
For $(e)$ $\Rightarrow$ $(a)$, repeat the second half of the  proof of Theorem \ref{7} part $(b)$.
\end{proof}


\section{Almost Dedekind ISP-domains}

In this section, we consider  the question whether any one-dimensional ISP-domain is an SP-domain. 
First, we recall some terminology from \cite{O}. Let $A$ be an almost Dedekind domain. The maximal ideals of $A$ containing a radical invertible ideal are called {\em non-critical}, while the others are called {\em critical}. Given $I$ an ideal of $A$ and $n\geq 1$, we set $V_n(I)=\{ M\in Max(A)\mid I\subseteq M^n\}$. Note that  $V_{n+1}(I)\subseteq V_n(I)$ and $V_1(I)$ is the usual Zariski closed set $V(I)$. Next, we recall \cite[Theorem 2.1]{O} and  add a new assertion $(g)$. 

\begin{theorem} {\em (\cite[Theorem 2.1]{O})}\label{24}
For an almost Dedekind domain $A$, the following assertions are equivalent.

$(a)$ $A$ is an SP-domain.

$(b)$ $A$ has no critical maximal ideals.

$(c)$ The radical of an invertible ideal is  invertible. 

$(d)$ Ever principal ideal is a product of radical ideals.

$(e)$ For every nonzero proper (principal) ideal $I$ and $n\geq 1$, the set $V_n(I)$ is (Zariski) closed in $Spec(A)$ and $V_m(I)$ is empty for some large $m$.

$(f)$ Every nonzero proper ideal $I$ can be factorized (uniquely)  as $I=J_1J_2\cdots J_n$ with radical ideals $J_1\subseteq J_2\subseteq \cdots \subseteq J_n$.

$(g)$ For every nonzero proper ideal $I$, we have $I=\sqrt{I}H$ for some ideal $H$.
\end{theorem}
\begin{proof}
 Since only $(g)$ is new,  it suffices to prove the equivalence of $(f)$ and $(g)$. 
 $(g)\Rightarrow (f)$ We have $I=\sqrt{I}H_1$ and $H_1=\sqrt{H_1}H_2$ for some ideals $H_1$ and $H_2$. Set $J_1=\sqrt{I}$ and $J_2=\sqrt{H_1}$, so $I=J_1J_2H_2$. From $I\subseteq H_1$, we get $J_1\subseteq J_2$.
 Repeating, we get $I=J_1J_2\cdots J_nH_n$ with  radical ideals $J_1\subseteq \cdots \subseteq J_n$. If some $H_n$ is $A$, we are done. If not, let $M$ be a maximal ideal  containing all $J_i$'s. Then $I=J_1J_2\cdots J_nH_n\subseteq M^n$ for each $n\geq 1$, which is a contradiction because $A_M$ is a DVR.
 Conversely, from $I=J_1\cdots J_n$ with $J_1\subseteq \cdots \subseteq J_n$ radical ideals, we get $\sqrt{I}=J_1$, so we are done. 
\end{proof}

In the next lemma, we recall two known facts.

\begin{lemma} \label{10}
If $A$ is an almost Dedekind domain which is not Dedekind, then:

$(a)$ Every noninvertible nonzero ideal of $A$   is contained in some noninvertible maximal ideal.

$(b)$ Every infinite  closed subset  of $Max(A)$ contains some  noninvertible maximal ideal.
\end{lemma}
\begin{proof}
$(a)$ is  a well-known application of Zorn's Lemma (every non finitely generated ideal is contained in a non finitely generated prime ideal).
$(b)$ Let $I$ be a nonzero ideal such that $V(I)$ is infinite.  By $(a)$, we may assume that $I$ is invertible, so the assertion follows from  \cite[Proposition 3.2.2]{FHL}. We give an alternative proof. For each $P\in V(I)$,
we have $IA_P=(PA_P)^{n_P}$ for some (unique) positive integer $n_P$. Consider the ideal $H=\sum_{P\in V(I)} IP^{-n_P}$. 
 It suffices to show that $H$ is not finitely generated, because $I\subseteq H$ implies $V(H)\subseteq V(I)$, so part $(a)$ applies. Suppose that $H$ is finitely generated. Then there exist distinct ideals $P_1,...,P_{k+1}\in V(I)$ such that 
 $IP_{k+1}^{-n_{k+1}}\subseteq \sum_{i=1}^k IP_i^{-n_i}$ where $n_j=n_{P_j}$.
Since the ideals $P_j$ are mutually comaximal, we have $IP_{k+1}^{-n_{k+1}}\subseteq I(\cap_{i=1}^k P_i^{n_i})^{-1}$, cf. \cite[Lemma 5.1]{Og}. We cancel $I$ and get $\cap_{i=1}^k P_i^{n_i}\subseteq P_{k+1}$, which is a contradiction.
\end{proof}


Recall that a domain $A$ {\em has  weak factorization}, if every nonzero nondivisorial ideal $I$ can be factored as the product of its divisorial closure $I_\nu$ and a finite product of maximal ideals; i.e., 
 $I=I_{\nu}M_1M_2\cdots M_n$ where $M_1$,$M_2$,...,$M_n$ are maximal ideals, cf. \cite{FHL1}.
By \cite[Proposition 4.2.14]{FHL}, an almost Dedekind domain $A$ has  weak factorization if and only if every nonzero  element of $A$ is contained in at most finitely many noninvertible maximal ideals.

Now let $A$ be an almost Dedekind domain $A$ which has  weak factorization.
Denote by $Z$ the set of noninvertible maximal ideals of $A$.
We introduce an ad-hoc concept: call an ideal $H$ of $A$ a {\em clean ideal}, if $H$ is invertible,  $V(H)\cap Z=\{M\}$  and $H\not\subseteq M^2$. 
Let $M\in Z$ and $f\in M-\{0\}$. By our hypothesis $V(f)\cap Z$ is finite, say equal to $\{M,M_1,...,M_n\}$. 
By  Prime Avoidance Lemma (e.g. \cite[Proposition 4.9]{G}),  
we can pick an element $g\in M-(M^2\cup M_1\cup \cdots \cup M_n)$, so $(f,g)$ is clean. Hence every $M\in Z$ contains a clean ideal. 
With terminology and notation above, we have:


\begin{theorem}\label{3}
For an almost Dedekind domain $A$ which has weak factorization, 
the following assertions are equivalent.
 
 $(a)$ $A$ is an SP-domain.
 
 $(b)$ $A$ is an ISP-domain.
 
 $(c)$ For every clean ideal $H$, the set $V_2(H)$ is finite.
 
 $(d)$ Every $M\in Z$ contains a clean ideal $H$ such that $V_2(H)$ is finite.
\end{theorem}
\begin{proof}
We may assume that $A$ is not a Dedekind domain.
Set $F=Max(A)-Z$.
$(a)$ $\Rightarrow$ $(b)$ is obvious.
$(b)$ $\Rightarrow$ $(c)$
Assume, to the contrary, that $H$ is a clean ideal and $V_2(H)$ contains an infinite set $\{ P_n\mid n\geq 1\}\subseteq F$.  Set $V(H)\cap Z=\{M\}$.
Let $I$ be the (integral) ideal $\sum_{n\geq 0}HP^{-1}_{2n+1}$. Since $H\subseteq  I$ and $V(H)\cap Z=\{M\}$, we get  $V(I)\cap Z=\{M\}$, because $M\supseteq H=P_{2n+1}HP^{-1}_{2n+1}$ implies $M\supseteq HP^{-1}_{2n+1}$. As $A$ is an ISP-domain, we can write $I=JQ$ with $J$ an invertible ideal and $Q\neq A$ a product of radical ideals. 
Since $M\in V(I)-V_2(I)$, we have one of the two   cases below.

{\em Case 1: $M\supseteq J$ and $M\not\supseteq Q$}.
Then $V(Q)\cap Z$ is empty, so  $Q$ is invertible, cf. Lemma \ref{10}. So $I=JQ$ is invertible, hence finitely generated. Then $HP^{-1}_{2n+1}\subseteq HP^{-1}_{1}+\cdots + HP^{-1}_{2n-1}$ for some $n\geq 1$. Since $H$ can be cancelled and the other ideals involved are invertible and comaximal, we get 
$P^{-1}_{2n+1}\subseteq (P_1\cap  \cdots \cap P_{2n-1})^{-1}$ (cf. \cite[Lemma 5.1]{Og}), hence $P_{2n+1}\supseteq P_1\cap  \cdots \cap P_{2n-1}$, which is a contradiction.

{\em Case 2: $M\not\supseteq J$ and $M\supseteq Q$}. 
Since $H\subseteq Q$ and $H\not\subseteq M^2$, we have that $V_2(Q)\cap Z=\emptyset$. As $Q$ is  a product of radical ideals, \cite[Lemma 1.10]{AD} shows  that $V_2(Q)$ is closed, so    $V_2(Q)$ is finite, cf. Lemma \ref{10}.
Note that $P_{2n}\in V_2(I)$ for every $n\geq 1$. 
Consequently, there exists some $m\geq 1$ such that 
$P_{2n}\in V(J)$ for each $n\geq m$. By Lemma \ref{10} and the fact that $H\subseteq J$, we get $V(J)\cap Z=\{M\}$, which is a contradiction.
 
$(c)$ $\Rightarrow$ $(d)$ is clear. 
$(d)$ $\Rightarrow$ $(a)$ By \cite[Theorem 2.1]{O}, it suffices to show that each  $M\in Z$ contains an invertible radical ideal. By $(d)$, $M$ contains a clean ideal $H$ such that $V_2(H)$ is finite, say equal to $\{ P_1,...,P_n\}$. 
For each $i$ between $1$ and $n$, we have $HA_{P_i}=P^{k_i}_iA_{P_i}$ for some $k_i\geq 2$. Then   $HP_1^{-k_1}\cdots P_n^{-k_n}$ is an invertible radical ideal contained  in $M$.
\end{proof}

The SP-domain $A$ constructed in \cite[Example 4.3]{O} has nonzero Jacobson radical and no $M\in Max(A)$  finitely generated. Thus $A$ does not have
 weak factorization.

\begin{corollary}\label{21}
Let $A$ be almost Dedekind domain having only finitely many noninvertible maximal ideals. Then $A$ is an ISP-domain if and only if $A$ is an SP-domain.
\end{corollary}

\begin{corollary}\label{22}
Let $A$ be an ISP-domain which has weak factorization and  $B$  a one-dimensional overring of $A$. Then $B$ is an SP-domain.
\end{corollary}
\begin{proof}
By Theorem \ref{7}, $A$ is a strongly discrete Pr\"ufer domain, so $B$ has weak factorization, cf. \cite[Corollary 4.3.3]{FHL}. Now apply Corollary \ref{11} and Theorem \ref{3}.
\end{proof}

The following question remains.

\begin{question}
 Is every one-dimensional ISP-domain an SP-domain ?
\end{question}

\section{An example}

In this final section we give an example of a two-dimensional ISP-domain $A$ which is not h-local. Hence $A$ is neither an SP-domain nor a ZPUI-domain.

\begin{proposition}\label{25}
Let $C$ be an SP-domain but not  Dedekind,  $M=qC$  a maximal principal ideal of $C$ and  $D$  a DVR   with quotient field $C/M$.
Assume there exists a unit $p$ of $C$ such that $\pi(p)$ generates the maximal ideal of $D$, where $\pi:C\rightarrow C/M$ is the canonical map.
Then the pull-back domain $A=\pi^{-1}(D)$ is a two-dimensional ISP-domain which is not h-local.
\end{proposition}
\begin{proof}
As $\pi(Mp^{-1})=0$, it follows that $M\subseteq pA$, so $A/pA$ is the residue field of $D$, because $A/M=D$ and $\pi(p)$ generates the maximal ideal of $D$.
Also, the only prime ideal of $A$ strictly containing $M$ is the maximal ideal $pA$.
By standard pull-back arguments (see for instance \cite[Lemma 1.1.4]{FHP}), the map $P\mapsto P\cap A$ is a bijection from $Spec(C)-V(M)$ to $Spec(A)-V(M)$ and $A_{P\cap A}=C_P$. By \cite[Corollary 1.1.9]{FHP}, $A$ is a two-dimensional Pr\"ufer domain. Also, by \cite[Lemma 1.1.6]{FHP}, we have $A[p^{-1}]=C[p^{-1}]=C$. Roughly speaking, $Spec(A)$ is obtained from $Spec(C)$ 
by adding the maximal ideal $pA\supseteq M$. Since $C$ is an almost Dedekind domain which is not  Dedekind, there exists a nonzero element $z\in A$ belonging to infinitely many maximal ideals of $A$, so $A$ is not h-local.
By \cite[Proposition 5.3.3]{FHP},
$B=A_{pA}$ is a two-dimensional strongly discrete valuation domain. It follows that $\cap_{t\geq 1} p^tA=M$. 

Let $I$ be an ideal of $A$. We observe that $I=IB\cap IC$. Indeed, if $N\in Max(A)-\{pA\}$, then $I\subseteq IC_{A-N}=IA_N$, so $IB\cap IC\subseteq \cap_{Q\in Max(A)}IA_Q=I$. In particular, we have $A=B\cap C$. Since $C$ is almost Dedekind and $M=qC$, we can write $IC=M^iJ$ where  $J$ is an ideal of $C$ with $M+J=C$ and $i\geq 0$, so $IC=M^i\cap J$. We also see that $H:=J\cap A\not\subseteq M$.
As $\cap_{t\geq 1} p^tA=M$,  we can write $H=p^jL=p^jA\cap L$ where  $L$ is an ideal of $A$ with $L\not\subseteq pA$ and $j\geq 0$. Consequently we get $$IC\cap A=M^i\cap J\cap A=M^i\cap H=M^i\cap p^jA\cap L$$
which equals either $M^i\cap L$ if $i\geq 1$ or $p^jA\cap L$ if $i=0$.
Using basic facts on valuation domains (see \cite[Section 17]{G}), it suffices to consider the following three cases. Each time we use the equality $I=(IB\cap  A) \cap (IC\cap A)$. 

Case 1: $IB=p^nB$ for some $n\geq 0$. We have $IB\cap A=p^nA$.
If $i\geq 1$, we get $I=p^nA\cap M^i\cap L=M^iL$. 
If $i=0$, we get $I=p^nA\cap p^jA\cap L=p^kL$ with $k=max(n,j)$.

Case 2: $IB=M^n$ for some $n\geq 1$. 
If $i\geq 1$, we get $I=M^n \cap M^i\cap L=M^kL$ with $k=max(n,i)$.
If $i=0$, we get $I=M^n\cap p^jA\cap L=M^nL$.

Case 3: $IB=p^nq^mA$ for some $m\geq 1$ and $n\in\mathbb{Z}$. We have $IB\cap A=p^nq^mA$, because $pA$ is the only maximal ideal containing $q$.
If $i>m\geq 1$, we get $I=p^nq^mA \cap M^i\cap L=M^iL$.
If $m\geq i\geq 1$, we get $I=p^nq^mA \cap M^i\cap L=p^nq^mL$.
If $i=0$, we get $I=p^nq^mA\cap p^jA\cap L=p^nq^mL$.

Consequently, to complete our proof, it suffices to show that $L$ is a product of radical ideals. Since $C$ is an SP-domain, we can write $LC=H_1\cdots H_n$ with each $H_i$ a radical ideal of $C$. Then each $J_i=H_i\cap A$ is a  radical ideal of $A$. Note that none of ideals $J_i$ is contained in $pA$, since $L\not\subseteq pA$. Set $R=J_1\cdots J_n$. Then  $R+pA=A$ and $L+pA=A$, so $R:p=R$ and $L:p=L$. Since  $RC=H_1\cdots H_n=LC$, we get $L=LC\cap A=RC\cap A=R$.
\end{proof}

Finally, we construct a specific domain satisfying the hypothesis of Proposition \ref{25}. We modify appropriately \cite[Example 3.4.1]{FHL}. 
If $A$ is a domain and $P_1$,...,$P_n$ are prime ideals of $A$, we denote by 
$A_{P_1\cup\cdots\cup P_n}$ the fraction ring of $A$ with denominators in 
$A-(P_1\cup\cdots\cup P_n)$.
Let $y$ and $(x_n)_{n\geq 1}$ be indeterminates over the rational field $\mathbb{Q}$. Consider the domain $$C=\bigcup_{n\geq 1} \mathbb{Q}[x_1,...,x_n,y/(x_1\cdots x_n)]_{(x_1)\cup \cdots \cup (x_n)\cup 
(y/(x_1\cdots x_n))}.$$
As $C$ is a union of an ascending chain of (semi-local) PID's, it is a one-dimensional Bezout domain. Adapting the proof of  \cite[Example 3.4.1]{FHL}, we see that the maximal ideals of $C$ are 
$N=\sum_{n\geq 1}(y/(x_1\cdots x_n))C$ 
and the principal ideals $(x_nC)_{n\geq 1}$. 
As $yC_M=MC_M$ for each $M\in Max(C)$, it follows that 
 $yC$ is a radical ideal, hence $N$ is non-critical. 
By \cite[Corollary 2.2]{O}, $C$ is an SP-domain. The residue field 
$C/x_1C$ is isomorphic to $K(y/x_1)$ where $K=\mathbb{Q}(x_n; {n\geq 2})$. Then $D=K[y/x_1]_{(y/x_1)}$ is a DVR with quotient field $C/x_1C$. Note that $x_1+y/x_1$ is a unit of $\mathbb{Q}[x_1,y/x_1]_{(x_1)\cup (y/x_1)}$, hence a unit of $C$. Moreover, the canonical map $C\rightarrow C/x_1C$ sends $x_1+y/x_1$ to $y/x_1$ which is a generator of the maximal ideal of $D$. Thus  $C$ satisfies the hypothesis of Proposition \ref{25}.
\\[2mm]

{\bf Acknowledgements.} 
The first author is highly grateful to Abdus Salam School of Mathematical Sciences Govt. Coll. University Lahore in supporting and facilitating this research.
The second author gratefully acknowledges the warm 
hospitality of the same institution during his  visits in the period 2006-2016.


\begin{thebibliography}{9999}

\bibitem{AD} M.T. Ahmed and T. Dumitrescu, SP-rings with zero-divisors, Comm. Algebra {\bf 45} (2017), 4435-4443.

\bibitem{Ak} T. Akiba, Remarks on generalized quotient rings, Proc. Japan Acad. {\bf 40} (1964), 801-806.

\bibitem{A} D.D. Anderson, Non-finitely generated ideals in valuation domains, Tamkang J. Math. {\bf 18} (1987), 49-52. 

\bibitem{CMZ}  D. Costa, J. Mott and M. Zafrullah, The construction $D+XD_S[X]$, J. Algebra {\bf 53} (1978), 423-439.


\bibitem{FHL1} M. Fontana, E. Houston and T. Lucas, Factoring ideals in Pr\"ufer domains, J. Pure Appl. Algebra {\bf 211}  (2007), 1-13.

\bibitem{FHL} M. Fontana, E. Houston and T. Lucas, {\em Factoring Ideals
in Integral Domains}, Springer 2013.

\bibitem{FHP} M. Fontana, J. Huckaba and I. Papick,
{\em Pr\"{u}fer Domains,}
Marcel Dekker, New York, 1997.

\bibitem{G} R. Gilmer, \textit{Multiplicative Ideal Theory}, Marcel Dekker, New York, 1972.

\bibitem{K}  I. Kaplansky, 
{\em Commutative Rings}, rev. ed., University of Chicago Press, Chicago, 1974.



\bibitem{O} B. Olberding, Factorization into radical ideals, in {\em Arithmetical Properties of Commutative Rings and Monoids} (S. Chapman, editor), Lect. Notes in Pure Appl. Math.   {\bf 241}, Chapman \& Hall, 363-377, 2005.

\bibitem{Og} B. Olberding,
{Globalizing local properties of Pr\"ufer domains,}
J. Algebra {\bf 205} (1998), 480-504.

\bibitem{OI} B. Olberding, Factorization into prime and invertible ideals, J. London Math. Soc.  {\bf 62} (2000), 336-344.

\bibitem{OII} B. Olberding, Factorization into prime and invertible ideals II, J. London Math. Soc. {\bf 80} (2009), 155-170.

\bibitem{R} F. Richman, Generalized quotient rings, Proc. Amer. Math. Soc. {\bf 16} (1965), 794-799. 

\bibitem{VY} N. Vaughan and R. Yeagy, Factoring ideals into semiprime ideals, Can. J. Math. {\bf 30} (1978), 1313-1318. 

\bibitem{Y} R. Yeagy, Semiprime factorizations in unions of Dedekind domains, J. Reine Angew. Math. {\bf 310} (1979), 182-186 



\end{thebibliography}
\end{document}